\documentclass{amsart}
\usepackage{art_styles2}

\begin{document}

\title[Embeddability of countably branching bundle graphs]{On the embeddability of countably branching bundle graphs into dual spaces}

\author{Y.~Perreau}
\address{Yoël Perreau, Laboratoire de Math\'ematiques de Besan\c con, Universit\'e Bourgogne Franche-Comt\'e, CNRS UMR-6623, 16 route de Gray, 25030 Besan\c con C\'edex, Besan\c con, France}
\email{yoel.perreau@univ-fcomte.fr}

\thanks{The author is supported by the French ``Investissements d'Avenir'' program, project ISITE-BFC (contract ANR-15-IDEX-03).}

\maketitle

\begin{abstract}

In this note the result from \cite{Swift} by A. Swift concerning the embeddability of countably branching bundle graphs into Banach spaces is extended from the context of reflexive spaces with an unconditional asymptotic structure to the context of dual spaces with a weak$^*$ unconditional asymptotic structure. We also investigate weak$^*$ asymptotic midpoint uniform convexity in dual spaces which turns out to be equivalent to its weak version in general and to the standard weak$^*$ asymptotic uniform convexity up to renorming in dual spaces with a weak$^*$ unconditional asymptotic structure.
\end{abstract}

\tableofcontents

\newpage 

\section*{Introduction}

The main result in the present paper in the following.

\begin{thmintro}\label{thmA}

Let $X$ be a separable Banach space and let $G$ be a non-trivial countably branching bundle graph. Let us assume that the Slzenk index $S_Z(X)$ is strictly greater than the first infinite ordinal $\omega$ and that the dual space $X^*$ has a weak$^*$ unconditional asymptotic structure.  Then every non-trivial countably branching bundle graph embeds Lipschitz in $X^*$ with distortion less than $6+\varepsilon$.

\end{thmintro}

Combined with previous results from \cite{Baudierandco} by F. Baudier and his co-authors this yields the following new characterization of weak$^*$ asymptotic uniform convexity up to renorming in dual spaces with a weak$^*$ unconditional asymptotic structure.

 \begin{thmintro}\label{thmB}

Let $X$ be a separable Banach space and let $G$ be a non-trivial countably branching bundle graph. Let us assume that $X^*$ has a weak$^*$ unconditional asymptotic structure. Then the following properties are equivalent.

\begin{enumerate}

\item{ The space $X$ admits an equivalent norm whose dual norm is AUC$^*$} 

\item{ The family $(G^{\otimes n})_{n\geq 1}$ of all countably branching bundle graphs generated by $G$ does not equi-Lipschitz embed into $X^*$. }

\end{enumerate}

\end{thmintro}

Those results fit in the asymptotic counterpart of the so called Ribe program which intends to characterize asymptotic properties of Banach spaces by purely metric ones. Interested readers are encouraged to look at the introduction in \cite{Baudierandco} for a detailed background around this kind of results.

The proof of the main result consists in pushing forward methods from \cite{Baudierandco} in order to give informations about the weak$^*$ asymptotic structure of dual spaces satisfying the conditions stated in the theorem and then in appealing to known embeddability results. More precisely we show that such a dual space must contain $\ell_\infty^n$ in its $n^{\text{th}}$ weak$^*$ asymptotic structure for every $n$ and this leads to the construction of a certain family of trees (the so called good-$\ell_\infty$ trees of arbitrary height) which is known by Swift's work in \cite{Swift} to allow for the embeddability of the graphs.

Taking a closer look at the proof it appears that one can forget the unconditionality condition by focusing on duals of separable spaces for which the unit ball contains sort of arbitrarily high multi-directional separated weak$^*$-convergent trees. This is the following property $(*)$.

\begin{align*}  (*):\ \  & \exists\varepsilon>0:\  \forall n\geq 1,\ \exists (x_s^*)_{s\in T_n}\subset X^* \text{ weak$^*$-null}:\ \ x^*_\emptyset\in B_{X^*},\ \forall s\neq \emptyset,\ \norme{x_s^*}\geq\varepsilon, \\
& \text{and } \forall s\in T_n,\ \abs{s}=n,\ \forall \varepsilon_1,\dots,\varepsilon_n\in\{-1, 1\},\ x^*_\emptyset+\sum_{i=1}^n \varepsilon_ix^*_{s\lvert i} \in B_{X^*}. \end{align*}

The result is then that the dual of a separable space with property $(*)$ must contain $\ell_\infty^n$ in its $n^{\text{th}}$ weak$^*$-asymptotic structure for all $n$ and thus is also a target space for the embeddability of families of bundle graphs. This observation lead us to introduce and to study a weak$^*$ asymptotic midpoint property on dual spaces (property AMUC$^*$). Our hope was that we would be able to develop a renorming theory to relate this property to the preceding one and then  obtain a nice metric characterisation of it in terms of embeddability of families of bundle graphs in the class of duals of separable spaces. But although it is straightforward to show that a dual with property AMUC$^*$ satisfies property $\neg(*)$ we are not able yet to obtain a renorming result that would give us the converse implication in full generality. Still we can restrict ourselves to the context of duals with a weak$^*$ unconditional asymptotic structure in which we can use the AUC$^*$ renorming theory to get the following linear result.
 
 \begin{thmintro}\label{thmC}

Let $X$ be a separable Banach space and let us assume that $X^*$ has a weak$^*$ unconditional asymptotic structure. Then the following properties are equivalent.

\begin{enumerate}

\item{ $S_Z(X)\leq \omega$ }

\item{ $X^*$ admits an equivalent AUC$^*$ norm} 

\item{ $X^*$ admits an equivalent AMUC$^*$ norm} 

\item{ $X^*$ satisfies property $\neg(*)$.}

\end{enumerate}

\end{thmintro}

Actually an obstruction to a more general renorming result for AMUC$^*$ is the lack of a dual property (a kind of AUS/AUC$^*$ duality) and the fact that we are not able to define a suitable derivation in order to find a replacement for the Szlenk index in this context.

Another goal of this study was to investigate the question of the equivalence of property AUC and property AMUC up to renorming. A corollary of the preceding theorem is that it holds for duals of separable spaces with a weak$^*$ unconditional asymptotic structure. The question of whether this structural condition is essential or not remains widely  open (even for reflexive spaces). By the way let us mention that is is not even known whether a Banach space $X$ with property AMUC has weak Szlenk index $w-S_Z(X)<\infty$ or not (in other word whether AMUC implies PCP or not) except again in the case it has an unconditional asymptotic structure. However we were able to get some information on the subject. We observed that properties AMUC and AMUC$^*$ are in fact equivalent on dual spaces. This is due to the tight relation between midpoint properties and the size of approximate midpoints (as stated in \cite{lenses} in the asymptotic context). Let us mention that this was already exploited by Kalton in his paper \cite{Kalton} where he tried to answer the question of the stability of property AUC under coarse-Lispchitz maps and that it was a also a key tool to implement a self-improvement argument in \cite{Baudierandco} in order to show non-embeddability of countably branching bundle graphs in AUC spaces. This observation has the following consequence: the dual of the James tree space $JT$ is AMUC$^*$. So since $JT$ admits no equivalent norm whose dual norm is AUC$^*$ the two properties are not equivalent up to renorming.

Let us discuss briefly the organisation of the paper. In section $1$ we introduce all the asymptotic structural properties we will need in the paper and we translate them in terms of properties of trees. In section $2$ we recall the definition of the Szlenk index $S_Z(X)$ of a Banach space $X$ and we show that under some separability and unconditional asymptotic structural assumptions, a Banach space $X$ such that $S_Z(X)>\omega$ satisfies $\ell_1^n\in \{X\}_n$ and $\ell_\infty^n\in \{X^*\}^{w^*}_n$ for every $n\geq 1$. In section $3$ we introduce property AMUC$^*$ and write down all the preceding comments concerning its properties. In section $4$ we recall the definition of bundle graphs and we prove the main result of the paper and his corollaries. We also give a small result on the embeddability of such graphs into Banach spaces with a $c_0$-spreading model. We conclude the paper with a section $5$ devoted to questions and comments. In particular we say a few words about the separability assumption in theorems $B$ and $C$ and we show that it is possible to get rid of it if we work with WCG Banach spaces by passing to a separable quotient.

\newpage

\section{Preliminaries: trees and asymptotic structure in dual spaces}

In this section we introduce several properties of dual spaces around the concept of weak$^*$ asymptotic structure and we express them in a suitable way for our later study in terms of behaviour of weak$^*$-null normalized trees. Although we might refer to it in the sequel we won't give the definitions of (weak) asymptotic structure properties and the corresponding expressions in terms of weak-null trees in general Banach spaces to avoid repetitions. Basically they are written in the same way as their weak$^*$ analogue but using the weak topology instead (that is mainly by replacing weak$^*$-null objects by weak null objects and weak$^*$-closed subspaces by closed subspaces) and replacing separability assumptions on the space by separability assumptions on the dual when needed. We refer to \cite{Baudierandco} for precise statements.

\subsection{Trees in dual spaces} Let us start by introducing the notion of trees in a dual space.

For all $N\geq 1$, let $T_N=\{\emptyset\}\cup\bigcup_{n=1}^N\N^n$ be the countably branching tree of height $N$ (without any metric consideration). Also let $T_\infty=\bigcup_{N\geq 1}T_N$. We will use the following notations. 

\begin{enumerate} 

\item{ For all $s=(s_1,\dots, s_n)\in T_\infty$ non empty, $\abs{s}:=n$ is the length of the sequence $s$, $s_{\lvert k}:=(s_1,\dots s_k)$ for every $1\leq k\leq n$ and $s_{\lvert 0}:=\emptyset$. Also, $s^-:=s_{\lvert n-1}$. }

\item{ For all $s=(s_1,\dots s_n)$ and $t=(t_1,\dots t_m)$ in $T_\infty$, $s\smallfrown t:=(s_1,\dots s_n,t_1,\dots, t_m)$. }

\end{enumerate}

A subtree $T$ of $T_N$ is a non empty subset of $T_N$ containing all the predecessors of its elements. A full subtree of $T_N$ is a subtree $T$ of $T_N$ such that all the sets $T\cap\{t\smallfrown n,\ n\in \N\}$ with $t\in T\cap T_{N-1}$ are infinite.

Let $T$ be a full subtree of $T_N$. For all $s,t\in T$, we write $s\leq t$ whenever the sequence $t$ is an extension of the sequence $s$ and we say that $s$ is an ancestor of $t$. This defines an ordering of $T$. We say that an enumeration $(\sigma_i)_{i\geq 1}$ of $T$ is a compatible linear ordering of $T$ if it satisfies the following condition: for every $i\geq 1$ and for every ancestor $s$ of $\sigma_i$, there is a $j<i$ such that $s=\sigma_j$.

\begin{rmk}

For all $s=(s_1,\dots,s_n)\in T$, let $f(s)=\sum_{i=1}^ns_i$ and let $f(\emptyset)=0$. If an enumeration $(\sigma_i)_{i\geq 1}$ of $T$ satisfies the condition: $f(\sigma_i)\leq f(\sigma_j)$ whenever $i\leq j$, then it is a compatible linear ordering of $T$. 

In the same way, if we replace $T$ by a full subtree of $[\N]^{\leq N}=\{s=(s_1,\dots, s_n)\in T_N:\ s_1<\dots<s_n\}$ and if we replace the function $f$ by the maximum in the preceding condition, we also get a compatible linear ordering of $T$.

It can be easily checked that all the results involving the notion of compatible linear ordering in this paper can be proved by using one of the two conditions above instead and thus that they are compatible with definitions and results either from \cite{Swift} or \cite{Baudierandco}.

\end{rmk}

Let $X$ be a Banach space and let $X^*$ be its dual space. A tree (of height $N$) in $X^*$ is a subset of $X^*$ indexed by a full subtree $T$ of $T_N$. Let $(x^*_s)_{s\in T}$ be a tree in $X^*$. We will use the following denominations. 

\begin{enumerate}

\item{ The root is the element $x^*_\emptyset$. An unrooted tree is a tree deprived of its root.}

\item{ A node is a set of the form $\{x^*_{t\smallfrown n},\ n\in \N,\ t\smallfrown n\in T\}$ with $t\in T\cap T_{N-1}$. }

\item{ A branch is a set of the form $\{x^*_\emptyset, x^*_{s_{\lvert 1}},\dots, x^*_s\}$  where $s\in T$ is a sequence of maximal length. If we are working with an unrooted tree, we omit $x^*_\emptyset$ in this definition. }

\item{ A leaf is an element $x^*_s$ with $s\in T$ of maximal length. }

\end{enumerate}

We say that the tree is bounded (respectively normalized) if the corresponding subset is bounded (respectively is contained in the unit sphere of $X^*$). We say that it is weak$^*$ null ($w^*$-null) if every node forms a weak$^*$ null sequence in $X^*$, that is to say if $w^*-\lim_n x^*_{t\smallfrown n}=0$ for all $t\in T$.

\subsection{Asymptotic structure of a dual space} Now let us talk about the weak$^*$ asymptotic structure of a dual space.

Let $X$ be a Banach space, let $X^*$ be its dual space and let $E$ be an $n$-dimensional Banach space with a normalized basis $(e_i)_{i=1}^n$ for some $n\geq 1$. Let us denote by $\wcof{X^*}$ the set of all weak$^*$ closed subspaces of $X^*$ of finite co-dimension. We say that $E$ is in the $n^{\text{th}}$ weak$^*$ asymptotic structure of $X^*$ up to a constant $C\geq 1$ if we have the following property:  \begin{align*}
\forall \varepsilon>0,\ \forall X_1\in \wcof{&X^*},\ \exists x^*_1\in  S_{X_1},\ \dots,\ \forall X_n\in \wcof{X^*},\ \exists x^*_n\in S_{X_n}: \\
&\text{ $(x^*_i)_{i=1}^n$ is $C(1+\varepsilon)$-equivalent to $(e_i)_{i=1}^n$. }
\end{align*}  If $C=1$, we say that $E$ is in the $n^{\text{th}}$ weak$^*$ asymptotic structure of $X^*$ and we write $E\in \{X^*\}^{w^*}_n$.

We say that $X^*$ has a weak$^*$ unconditional asymptotic structure with constant $C\geq 1$ if it satisfies:  \begin{align*}
\exists X_1\in \wcof{X^*},\ \forall x^*_1&\in  S_{X_1},\ \dots,\ \exists X_n\in \wcof{X^*},\ \forall x^*_n\in S_{X_n}: \\
&\text{ $(x^*_i)_{i=1}^n$ is $C$-unconditional. } 
\end{align*}

\begin{rmk}

This properties can be reformulated by using the langage of games between two players. Here one player picks weak$^*$ closed subspaces of $X^*$ of finite co-dimension and the other player picks normalized vectors in those subspaces. Detailed explanitions and references are given in \cite{Baudierandco} paragraph 3.2.

\end{rmk}

\subsection{Tree reformulation for duals of separable spaces}

Finally, let us explain how the weak$^*$ asymptotic structure properties introduced before can be translated in duals of separable spaces in terms of properties of weak$^*$ null normalized trees. We believe that the following results are well known but we give the proof for completeness.

\begin{lm}

Let $X$ be a separable Banach space. For all $n\geq 1$ and for all $n$-dimensional space $E$ with a normalized basis $(e_i)_{i=1}^n$, $E$ is in the $n^{\text{th}}$ weak$^*$ asymptotic structure of  $X^*$ up to a constant $C\geq 1$ if and only if for all $\varepsilon >0$, there is a weak$^*$-null normalized unrooted tree of height $n$ whose branches are all $C(1+\varepsilon)$-equivalent to $(e_i)_{i=1}^n$.

\end{lm}

\begin{proof}

Let us fix some $n\geq 1$ and some $n$-dimensional space $E$ with a normalized basis $(e_i)_{i=1}^n$.

First let us assume that $E$ is in the $n^{\text{th}}$ weak$^*$ asymptotic structure of $X^*$ up to a constant $C\geq 1$.  Since $X$ is separable, we can choose a dense sequence $(z_k)_{k\geq 1}$ in $S_X$. For all $k\geq 1$, let us define the weak$^*$ closed subspace of finite codimension $Z_k=\bigcap_{i=1}^k\ker z_i$ of $X^*$ where $z_i$ is identified with the corresponding element in $X^{**}$. Then we have the property that any normalized sequence $(z_k^*)_{k\geq 1}$ in $S_{X^*}$ such that $z_k^*\in Z_k$ for all $k\geq 1$ is weak$^*$-null. Let us fix $\varepsilon>0$. Using the assumption on $E$, we will build a normalized unrooted tree $(z_s^*)_{s\in T_n\backslash\{\emptyset\}}$ in $S_{X^*}$ such that:
\begin{enumerate}

\item{ For all $s\in T_{n-1}$, for all $k\in \N$, $z_{s\smallfrown k}^*\in Z_k$. }

\item{ All the branches of $(z_s^*)_{s\in T_n\backslash\{\emptyset\}}$ are $C(1+\varepsilon)$-equivalent to $(e_i)_{i=1}^n$. }

\end{enumerate} This gives the desired tree by the choice of the subspaces $Z_k$. To do this construction, it is convenient to introduce a compatible linear ordering $(\sigma_i)_{i\geq 1}$ of $T_n\backslash\{\emptyset\}$. Then a straightforward induction allows us to get a sequence $(z_{\sigma_i}^*)_{i\geq 1}$ in $S_{X^*}$ such that:
\begin{enumerate}

\item{ If $\sigma_i=\sigma_i^-\smallfrown k$ for some $k\in \N$ then $z_{\sigma_i}\in Z_k$. }

\item{ If $l=\abs{\sigma_i}=n$, then $(z_\sigma^*,\ \sigma\leq \sigma_i)$ is $C(1+\varepsilon)$-equivalent to $(e_i)_{i=1}^n$.  }

\item{ If $l=\abs{\sigma_i}<n$ it satisfies:  \begin{align*}
\forall X_{l+1}\in \wcof{&X^*},\ \exists x^*_{l+1}\in  S_{X_{l+1}},\ \dots,\ \forall X_n\in\wcof{X^*},\ \exists x^*_n\in S_{X_n}: \\
(z_\sigma^*,\ \sigma&\leq \sigma_i)\smallfrown (x^*_i)_{i=l+1}^n \text{ is $C(1+\varepsilon)$-equivalent to $(e_i)_{i=1}^n$. }
\end{align*}  }

\end{enumerate}

Second, let us fix $\varepsilon>0$ and let us take $\delta>0$ such that: if $(v_i)_{i=1}^n$ and $(w_i)_{i=1}^n$ are normalized sequences in $X^*$, if $(v_i)_{i=1}^n$ is $C(1+\delta)$-equivalent to $(e_i)_{i=1}^n$ and if $\norme{v_i-w_i}\leq \delta$ for all $i\leq n$, then $(w_i)_{i=1}^n$ is $C(1+\varepsilon)$-equivalent to $(e_i)_{i=1}^n$. Then let us assume that there is a weak$^*$-null normalized unrooted tree $(z_s^*)_{s\in T_n\backslash\{\emptyset\}}$ in $S_{X^*}$ whose branches are all $C(1+\delta)$-equivalent to $(e_i)_{i=1}^n$. Let us recall the following well known property: if $Z$ is a weak$^*$ closed subspace of finite codimension of $X^*$ and if $(z_k^*)_{k\geq 1}$ is normalized and weak$^*$-null, then the distance $\dist(z_k^*,S_Z)$ goes to $0$. Using this, a straightforward induction shows that the following property holds for all $1\leq j\leq n$: \begin{align*}
 \forall X_1\in\cof{X^*},\ \exists x^*_{1}\in  S_{X_1},\ &\dots,\ \forall X_j\in\cof{X^*},\ \exists x^*_j\in S_{X_j},\ \exists s\in T_n,\ \abs{s}=j: \\
 &\forall 1\leq i\leq j,\ \norme{x_i^*-z_{s_{\lvert i}}^*}\leq \delta 
\end{align*} where the $X_i$ are weak$^*$ closed subspaces of finite co-dimension of $Y$. By our choice of $\delta$ and by the properties of the branches of our tree, we get the desired result.

\end{proof}

Let us recall two famous results in Banach space theory.

\begin{thm}[James, non-distordability of $c_0$] 

For all $m\geq 1$, $C\geq 1$ and $\varepsilon>0$, there is an $n\geq 1$ such that: any basic sequence $(e_i)_{i=1}^n$ of length $n$ which is $C$-equivalent to the unit vector basis of $\ell_\infty^n$ admits a block basis of length $m$ which is $(1+\varepsilon)$-equivalent to the unit vector basis of $\ell_\infty^m$.

\end{thm}

\begin{thm}[Krivine/Rosenthal, non-distordability of $\ell_p$] 

Let $1\leq p <\infty$. For all $m\geq 1$, $C\geq 1$ and $\varepsilon>0$, there is an $n\geq 1$ such that: any basic sequence $(e_i)_{i=1}^n$ of length $n$ which is $C$-equivalent to the unit vector basis of $\ell_p^n$ admits a block basis of length $m$ which is $(1+\varepsilon)$-equivalent to the unit vector basis of $\ell_p^m$.

\end{thm}

We will use these results and the preceding lemma to optimize constants in the asymptotic structure.

\begin{lm}

Let $X$ be a separable Banach space and let $1\leq p \leq \infty$. If there is a constant $C\geq 1$ such that the space $\ell_p^n$ is in the $n^{\text{th}}$ weak$^*$ asymptotic structure of $X^*$ up to $C$ for all $n\geq 1$, then $\ell_p^n\in\{X^*\}^{w^*}_n$ for all $n\geq 1$.

\end{lm}

\begin{proof}

Let us assume that there is a constant $C>0$ such that for every $n\geq 1$, $\ell_p^n$ is in the $n^{\text{th}}$ weak$^*$ asymptotic structure of $X^*$ up to $C$. First we will use a stabilisation argument to show that for every $n\geq 1$ we can find an $n$-dimensional normed space $E_n$ which is $2C$ equivalent to $\ell_p^n$ and belongs to $\{X^*\}_n^{w^*}$.

Let us fix $n\geq 1$ and  $\varepsilon \in (0,1)$. By compactness of the Banach-Mazur distance on the set of $n$-dimensional normed spaces we can find a finite collection $\mathcal{E}_\varepsilon$ of such spaces such that every $n$-dimensional normed space which is $2C$ equivalent to $\ell_p^n$ is actually $(1+\varepsilon)$ equivalent to some element of $\mathcal{E}_\varepsilon$. Since $X$ is separable we can find by the preceding lemma a weak$^*$ null normalized unrooted tree $(x_s^*)_{s\in T_n\backslash\{\emptyset\}}$ in $X^*$ whose branches are all $2C$ equivalent to the canonical basis of $\ell_p^n$. Thus for every branch $\beta$ of the tree there is an element $E_\beta\in \mathcal{E}_\varepsilon$ which is $(1+\varepsilon)$ equivalent to the span of $\beta$. By Ramsey's theorem we may assume up to the extraction of a full subtree that there is an element $E_\varepsilon\in \mathcal{E}_\varepsilon$ such that the span of every branch $\beta$ of the tree is $(1+\varepsilon)$ equivalent to $E_\varepsilon$. Finally letting $\varepsilon$ tends to $0$ and using sequential compactness we can find a sequence $(\varepsilon_i)_{i\geq 1}$ in $(0,1)$ converging to $0$ and an $n$-dimensional normed space $E$ such that the sequence $(E_{\varepsilon_i})_{i\geq 1}$ converges to $E$ in Banach-Mazur distance. It is straightforward to check that $E_n=E$ has the two desired properties.

Second we apply the non-distordability results to get the desired conclusion. Fix $m\geq 1$ and $\varepsilon>0$. Then we can find an $n\geq 1$ such that every normalized basic sequence of length $n$ which is $2C$ equivalent to the canonical basis of $\ell_p^n$ admits a normalized block basis of length $m$ which is $(1+\varepsilon)$ equivalent to the canonical basis of $\ell_p^m$. Now use the first step to find an $n$-dimensional normed space $E_n$ with a normalized basis $(e_i)_{i= 1}^n$ which is $2C$ equivalent to $\ell_p^n$ and belongs to $\{X^*\}_n^{w^*}$ and use the preceding statement to pick a normalized block basis $(f_j)_{j=1}^m$ of  $(e_i)_{i= 1}^n$ which is $(1+\varepsilon)$ equivalent to the canonical basis of $\ell_p^m$. Decomposing the $f_j$'s in the original basis and following the idea that one can play the game with the $e_i$'s taking successively the same subspace one can use a perturbation argument in order to check that the following property holds: \begin{align*}
\forall X_1\in \cof{&X^*},\  \exists x^*_1\in  S_{X_1},\  \dots,\ \forall X_n\in \cof{X^*},\ \exists x^*_m\in S_{X_n}: \\
&\text{ $(x^*_i)_{i=1}^m$ is $(1+\varepsilon)$-equivalent to $(f_i)_{i=1}^m$. }
\end{align*} As a consequence all the sequences $(x^*_i)_{i=1}^m$ are $(1+\varepsilon)^2$ equivalent to the canonical basis of $\ell_p^m$ and the conclusion follows by letting $\varepsilon$ go to $0$.

\end{proof}

Let $X$ be a Banach space and let $C\geq 1$. We say that $X^*$ has the $C$-weak$^*$ unconditional finite tree property if any weak$^*$-null normalized unrooted tree of finite height in $X^*$ has a branch which is $C$-unconditional. Let us point out that in this case an application of Ramsey's combinatorial theorem yields that any  weak$^*$-null normalized unrooted tree of finite height in $X^*$ has a full subtree whose branches are all $C$-unconditional. This property can be related to weak$^*$ asymptotic unconditionally for duals of separable spaces.

\begin{lm}

Let $X$ be a separable Banach space. Then $X^*$ has a weak$^*$ unconditional asymptotic structure if and only if it has the weak$^*$ unconditional finite tree property.

\end{lm}

\begin{proof}

Clearly one can show using the same tools as in the proof of the preceding lemma that:\begin{enumerate}
\item{ If $X^*$ does not have a $C$-weak$^*$ unconditional asymptotic structure then there is a weak$^*$-null normalized unrooted tree of height $n$ in $X^*$ with no $C$-unconditional branch. }
\item{ If $X^*$ has a $C$-weak$^*$ unconditional asymptotic structure then for all $\varepsilon>0$, and for all weak$^*$-null normalized unrooted tree of height $n$ in $S_{X^*}$ there is a branch which is $(C+\varepsilon)$-unconditional. }

\end{enumerate}

\end{proof}

\section{Szlenk index and asymptotic structure}

In this section we introduce a key object for our study, the Szlenk index of a Banach space, and we investigate how this object affects the asymptotic structure of Banach spaces and of their duals under some separability assumptions and unconditionality conditions.

\subsection{Szlenk index and bi-orthogonal systems indexed by trees}

Let $X$ be a Banach space, let $K$ be a weak$^*$-compact subset of the dual space $X^*$ and fix $\varepsilon>0$.  Denote by $\mathcal{V}$ the set of all weak$^*$-open subsets $V$ of $K$ satisfying $\diam\ V \leq \varepsilon$ and let $s_\varepsilon(K)$ be the weak$^*$ closed subset defined by $s_\varepsilon(K)=K\backslash\left(\bigcup_{V\in \mathcal{V}}V\right)$. We define inductively derivative subsets $s_\varepsilon^\alpha(K)$ of  $K$ for every ordinal $\alpha$ by $s_\varepsilon^1(K)=s_\varepsilon(K)$, $s_\varepsilon^{\alpha+1}(K)=s_\varepsilon(s_\varepsilon^{\alpha}(K))$ if $\alpha\geq 1$ and $s_\varepsilon^{\alpha}=\bigcap_{\beta<\alpha}s_\varepsilon^{\beta}$ if $\alpha$ is a limit ordinal. Then we consider $S_Z(K,\varepsilon)=\inf\edtq{\alpha}{s_\varepsilon^{\alpha}(K)=\emptyset}$ if such an $\alpha$ exists and $S_Z(K,\varepsilon)=\infty$ otherwise.  Finally let $S_Z(K)=\sup_{\varepsilon>0} S_Z(K,\varepsilon)$. The Szlenk index of $X$ is $S_Z(X)=S_Z\left(B_{X^*}\right)$ where $B_{X^*}$ denotes the closed unit ball of $X^*$. We refer to the survey \cite{surveySzlenk} for an extensive study of the properties  and applications of the Slzenk index.

In the separable case it is well known that the property $S_Z(X)>\omega$ can be characterized by the existence in the unit ball of the dual space $X^*$ of arbitrarily high separated weak$^*$ null trees. This is the following lemma. 

\begin{lm}\label{Szlenktrees}

Let $X$ be a separable space. Then $S_Z(X)> \omega$ if and only if $X^*$ has the following property: there is an $\varepsilon>0$ such that for every integer $n\geq 1$ there is a weak$^*$-null tree $(x_s^*)_{s\in T_n}$ in $X^*$ with $x_\emptyset^*\in B_{X^*}$ and $\norme{x_s^*}\geq \varepsilon$ for all $s\in T_n$ non-empty, such that for every sequence $s\in T_n$ of maximal length $x_\emptyset^*+ \sum_{i=1}^n x_{s_{\lvert i}}^*$ is in $B_{X^*}$.

\end{lm}

In \cite{BKL}, the authors used the preceding lemma to prove the following proposition which we will exploit later.

\begin{prp}[\cite{BKL}, prop 2.2.]\label{propBKL}

Let $X$ be a separable Banach space with $S_Z(X)>\omega$. For all $n\geq 1$ and for all $\delta>0$, there exist a weak$^*$-null bounded tree $(x_s^*)_{s\in T_n}$ in  $X^*$ and a normalized tree $(x_s)_{s\in T_n}$ in $X$ such that:
\begin{enumerate}

\item{ For all $s\in T_n\backslash\{\emptyset\}$, $\norme{x_s^*}\geq 1$ and $\norme{\sum_{t\leq s}x_t^*}\leq 3$. }

\item{ For all $s\in T_n$, $x_s^*(x_s)\geq \frac{1}{3}\norme{x_s^*}$. }

\item{ For all $s\neq t$ in $T_n$, $\abs{x_s^*(x_t)}\leq \delta$. }

\end{enumerate}

\end{prp}

The proof of this result rests on the following classical result from Mazur.

\begin{lm}[Mazur]

Let $(x_k^*)_{k\geq 1}$ be a weak$^*$-null sequence in $X^*$ such that $\norme{x_k^*}\geq 1$ for all $k\geq 1$, and let $F$ be a finite subset of $X^*$. Then there is a sequence $(x_k)_{k\geq 1}$ in $S_X$ such that $x_k\in \bigcap_{y^*\in F}\ker y^*$ for all $k\geq 1$ and $\liminf x^*_k(x_k)\geq \frac{1}{2}$.

\end{lm}

This proposition can be slightly improved when we assume that the dual space is separable: one can ask the normalized trees in $S_X$ to be weakly-null. 

\begin{prp}\label{propBKL+}

Let $X$ be a Banach space with separable dual and with $S_Z(X)>\omega$. For all $n\geq 1$ and for all $\delta$, there exist a weak$^*$-null bounded tree $(x_s^*)_{s\in T_n}$ in  $X^*$ and a normalized weak-null tree $(x_s)_{s\in T_n}$ in $X$ such that:
\begin{enumerate}

\item{ For all $s\in T_n\backslash\{\emptyset\}$, $\norme{x_s^*}\geq 1$ and $\norme{\sum_{t\leq s}x_t^*}\leq 3$. }

\item{ For all $s\in T_n$, $x_s^*(x_s)\geq \frac{1}{3}\norme{x_s^*}$. }

\item{ For all $s\neq t$ in $T_n$, $\abs{x_s^*(x_t)}\leq \delta$. }

\end{enumerate}

\end{prp}

\begin{proof}

Let us pick a dense sequence $(y_k^*)_{n\geq 1}$ in $S_{X^*}$ and let us define $Y_k=\bigcap_{i=1}^k \ker y_i^*$. Then any normalized sequence $(x_k)_{k\geq 1}$ in $X$ such that $x_k\in Y_k$ for all $k\geq 1$ is weakly-null.  Now let us fix $n\geq 1$ and $\delta>0$. Since $S_Z(X)>\omega$ and $X$ is separable we can find a weak$^*$-null bounded tree $(x_s^*)_{s\in T_n}$ in $X^*$ such that for all $s\in T_n\backslash\{\emptyset\}$, $\norme{x_s^*}\geq 1$ and $\norme{\sum_{t\leq s}x_t^*}\leq 3$.  Let us pick a compatible linear ordering $(\sigma_i)_{i\geq 1}$ of $T_N$ and let us assume as we may that this ordering also satisfies the following property: if $\sigma_i=\sigma_{i_0}\smallfrown m$ and $\sigma_j=\sigma_{i_0}\smallfrown m'$ with $m<m'$ then $i<j$.  By induction, we can build a sequence $(\theta_i)_{i\geq 1}$ in $T_n$ and a sequence $(x_{\theta_i})_{i\geq 1}$ in $S_X$ such that:
 \begin{enumerate}
 
 \item{ $\theta_1=\emptyset$ }
 
 \item{ If $\sigma_i^-=\sigma_{i_0}$ then $\theta_i=\theta_{i_0}\smallfrown m'$ for some $m'\in\N$ bigger than $\max\{l:\ \exists j<i:\ \theta_j=\theta_{i_0}\smallfrown l\}$ and $x_{\theta_i}\in Y_m$. }
 
 \item{ For all $i<j$, $\abs{x_{\theta_j}^*(x_{\theta_i})}\leq \delta$ and $x_{\theta_i}^*(x_{\theta_j})=0$. } 
 
 \item{ For all $i\geq 1$, $x_{\theta_i}^*(x_{\theta_i})\geq \frac{1}{3}\norme{x_{\theta_i}^*}$. }
 
 \end{enumerate} For $i=1$, just let $\theta_1=\emptyset$ and pick any $x_{\emptyset}$ in $S_X$ such that $x_{\emptyset}^*(x_{\emptyset})\geq \frac{1}{3}\norme{x_{\emptyset}^*}$. Now let us assume that $\theta_1,\dots \theta_i$ and $x_{\theta_1},\dots, x_{\theta_i}$ have been chosen with the required properties. Since we are working with a compatible linear ordering, there is some $i_0\leq i$ and such that $\sigma_{i+1}^-=\sigma_{i_0}$. Since $(x_{\theta_{i_0}\smallfrown k}^*)_{k\geq 1}$ is weak$^*$-null there is a $K\in \N$  such that $\abs{x_{\theta_{i_0}\smallfrown k}^*(x_{\theta_j})}\leq \delta$ for all $j\leq i$ and for all $k\geq K$. Moreover, we can apply Mazur's lemma to this sequence with $F=\{y_1^*,\dots, y_m^*, x_{\theta_1}^*,\dots, x_{\theta_i}^*\}$ and take any $m'$ big enough in order to get a element $x_{\theta_{i_0}\smallfrown m'}\in S_{Y_m}$ such that $x_{\theta_{i_0}\smallfrown m'}^*(x_{\theta_{i_0}\smallfrown m'})\geq \frac{1}{3}\norme{x_{\theta_{i_0}\smallfrown m'}^*}$ and $x_j^*(x_{\theta_{i_0}\smallfrown m'})=0$ for all $j\leq i$ and also $m'$ bigger than $K$ and bigger than $\max\{l:\ \exists j<i:\ \theta_j=\theta_{i_0}\smallfrown l\}$. We conclude by putting $\theta_{i+1}=\theta_{i_0}\smallfrown m'$. Then $T=\{\theta_i,\ i\geq 1\}$ defines a full subtree of $T_n$ and for every node $(x_{\theta_{i_0}\smallfrown m'_k})_{k\geq 1}$ of $T$ there is a strictly increasing sequence $(m_k)_{k\geq 1}$ in $\N$ such that $x_{\theta_{i_0}\smallfrown m'_k}\in S_{Y_{m_k}}$. Consequently, $(x_s)_{s\in T}$ is a weakly-null tree in $X$ and satisfies all the required properties together with $(x_s^*)_{s\in T}$.

\end{proof}

\subsection{Weak$^*$ asymptotic structure and Szlenk index }

Using the almost bi-orthogonal system from \cite{BKL}, one can prove the following result, which is an extension of the reflexive version proved in \cite{Baudierandco}.

\begin{thm}

Let $X^*$ be the dual of a separable Banach space $X$ with $S_Z(X)>\omega$, and assume that $X^*$ has a weak$^*$ unconditional asymptotic structure. Then $\ell_\infty^n$ is in the $n^{\text{th}}$ weak$^*$ asymptotic structure of $X^*$ for every $n\geq 1$.

\end{thm}

\begin{proof}

By the results of the first section there is a constant $C\geq 1$ such that any weak$^*$-null normalized tree of finite height in $X^*$ has a full subtree whose branches are all $C$ unconditional. Moreover it is sufficient to show that there is a constant $D\geq 1$ such that for all $n\geq 1$, there is a weak$^*$ null normalized tree of height $n$ in $Y$ whose branches are all $D$-equivalent to the unit vector basis of $\ell_\infty^n$.  Let us fix some $n\geq 1$ and let us fix $\delta>0$ to be chosen later. By proposition \ref{propBKL}, there exist a weak$^*$-null bounded tree $(x_s^*)_{s\in T_n}$ in  $X^*$ and a normalized tree $(x_s)_{s\in T_n}$ in $X$ such that:
\begin{enumerate}

\item{ For all $s\in T_n\backslash\{\emptyset\}$, $\norme{x_s^*}\geq 1$ and $\norme{\sum_{t\leq s}x_t^*}\leq 3$. }

\item{ For all $s\in T_n$, $x_s^*(x_s)\geq \frac{1}{3}\norme{x_s^*}$. }

\item{ For all $s\neq t$ in $T_n$, $\abs{x_s^*(x_t)}\leq \delta$. }

\end{enumerate} As mentioned one can assume, up to passing to some full subtree, that all banches of $(x_s^*)_{s\in T_n}$ are $C$-unconditional. For all $s\in T_n$, let $y_s=\frac{x_s^*}{\norme{x_s^*}}$. Since the function $(a_0,\dots, a_{\abs{s}})\in[-1,\ 1]^{\abs{s}} \mapsto  \norme{\sum_{t\leq s}a_tx_t^*}$ is continuous and convex we have:  \begin{align*}
\max \left\{ \norme{\sum_{t\leq s}a_ty_t},\ a_t\in [-1,\ 1] \right\} &\leq \max \left\{ \norme{\sum_{t\leq s}a_tx_t^*},\ a_t\in [-1,\ 1] \right\} \\
&= \max \left\{ \norme{\sum_{t\leq s}a_tx_t^*},\ a_t\in \{-1,\ 1\} \right\} \\
&\leq C\norme{\sum_{t\leq s}x_t^*} \\
&\leq 3C.
\end{align*} Thus, for all $s\in T_n$ with $\abs{s}=n$ and for all $a_0,\dots, a_n\in \R$, we have: $$\norme{\sum_{t\leq s}a_{\abs{t}}y_t}\leq 3C\max_{0\leq i\leq n} \abs{a_i}.$$ Now if $j$ is such that $\abs{a_j}=\max_{0\leq i\leq n} \abs{a_i}$, we have, using the properties of the almost bi-orthogonal system: \begin{align*}
\norme{\sum_{t\leq s}a_{\abs{t}}y_t} &\geq \scal{\sign {a_j}x_{s_{\lvert j}}}{\sum_{t\leq s}a_{\abs{t}}y_t} \\
&\geq  \frac{1}{3}\abs{a_j} -\delta \sum_{i\neq j}\abs{a_i} \\
&\geq \left(\frac{1}{3}-n\delta\right) \max_{0\leq i\leq n} \abs{a_i}.
\end{align*}  Thus if $\delta$ was chosen say smaller than $\frac{1}{6n}$, the tree $(y_s)_{s\in T_n}$ satisfies the required conditions (with constant $D=18C$).

\end{proof}

To conclude this paragraph let us make a brief comment about the property $(*)$ mentioned in the introduction. Let us first recall the definition of this property: \begin{align*}  (*):\ \  & \exists\varepsilon>0:\  \forall n\geq 1,\ \exists (x_s^*)_{s\in T_n}\subset X^* \text{ weak$^*$-null}: \ x^*_\emptyset\in B_{X^*},\ \norme{x_s^*}\geq\varepsilon\ \forall s\neq \emptyset,\ \\
& \forall s\in T_n,\ \abs{s}=n,\ \forall \varepsilon_1,\dots,\varepsilon_n\in\{-1, 1\},\ x^*_\emptyset+\sum_{i=1}^n \varepsilon_ix^*_{s\lvert i} \in B_{X^*}. \end{align*} By the expression of the Szlenk index in terms of containment of trees, is is clear that if $X$ is a separable space, then $(*)$ $\implies$ $S_Z(X)>\omega$. Moreover one easily sees that the converse is true if one adds the condition that $X^*$ has a weak$^*$unconditional asymptotic structure. In fact it seems that property $(*)$ is the right property to make the preceding proof work so we have the following result.

\begin{thm}

Let $X$ be a separable Banach space and let us assume that $X^*$ has property $(*)$. Then $\ell_\infty^n$ is in the $n^{\text{th}}$ weak$^*$ asymptotic structure of $X^*$ for every $n\geq 1$.

\end{thm}

This property $(*)$ remains quite mysterious and although having a formally stronger result we don't known wether a dual space satisfying this property has a weak$^*$ unconditional asymptotic structure or not.  We will see some more things related to it in the following section.

\subsection{Asymptotic structure and Szlenk index}

To conclude this section we show how the small improvement of the result of \cite{BKL} from the first paragraph concerning almost bi-orthogonal systems allows to prove the pendant of the structural result of the preceding paragraph in the space $X$ with an $\ell_1$ flavour.

\begin{thm}

Let $X$ be a Banach space with a separable dual, with an unconditional asymptotic structure and with $S_Z(X)>\omega$. Then $\ell_1^n$ is in the $n^{\text{th}}$ asymptotic structure of $X$ for every $n\geq 1$.

\end{thm}

\begin{proof}

By the weak version from \cite{Baudierandco} of the results of the first section there is a constant $C\geq 1$ such that any weak-null normalized tree of finite height in $X$ has a full subtree whose branches are all $C$ unconditional. Moreover it is sufficient to show that there is a constant $D\geq 1$ such that for all $n\geq 1$, there is a weak-null normalized tree of height $n$ in $X$ whose branches are all $D$-equivalent to the unit vector basis of $\ell_1^n$. Let us fix some $n\geq 1$ and let us fix $\delta>0$ to be chosen latter. By proposition \ref{propBKL+}, there is a weak$^*$-null bounded tree $(x_s^*)_{s\in T_n}$ in  $X^*$ and a normalized weak-null tree $(x_s)_{s\in T_n}$ in $X$ such that:
\begin{enumerate}

\item{ For all $s\in T_n\backslash\{\emptyset\}$, $\norme{x_s^*}\geq 1$ and $\norme{\sum_{t\leq s}x_t^*}\leq 3$. }

\item{ For all $s\in T_n$, $x_s^*(x_s)\geq \frac{1}{3}\norme{x_s^*}$. }

\item{ For all $s\neq t$ in $T_n$, $\abs{x_s^*(x_t)}\leq \delta$. }

\end{enumerate} As mentioned one can assume, up to passing to some full subtree, that all branches of $(x_s)_{s\in T_n}$ are $C$-unconditional. Since $(x_s)_{s\in T_N}$ is normalized, a simple triangular inequality yields: $$\norme{\sum_{t\leq s}a_{\abs{t}}x_t}\leq \sum_{i=0}^n\abs{a_i}$$ for all $a_0,\dots, a_n\in\R$ and for all $s\in T_n$ of length $n$.

Moreover, \begin{align*}
\norme{\sum_{t\leq s}a_{\abs{t}}x_t} &\geq \frac{1}{C}\norme{\sum_{t\leq s}\sign{a_{\abs{t}}}a_{\abs{t}}x_t} \\
&\geq \frac{1}{3C} \scal{\sum_{t\leq s}\abs{a_{\abs{t}}}x_t}{\sum_{t\leq s}x_t^*} \\
&\geq \left(\frac{1}{9C}-n\delta\right)\sum_{i=0}^n\abs{a_i} \\
&\geq \frac{1}{D} \sum_{i=0}^n\abs{a_i}
\end{align*} for any chosen constant $D>9C$ if $\delta$ was chosen small enough.

\end{proof}

\section{Weak$^*$ asymptotic midpoint convexity in dual spaces}

In this section we introduce a weak$^*$ midpoint asymptotic convexity property (property AMUC$^*$) and begin a systematic study of it in comparison with the well studied property AUC$^*$.

\subsection{Properties AUC$^*$ and AMUC$^*$}

Let $X$ be a Banach space and let $X^*$ be its dual space. For all $t>0$ and for all $x^*\in S_{X^*}$, let  $$\modconvunifw{X}{t,x^*}=\sup_{Y\in\wcof{X^*}}\inf_{y^*\in S_Y}(\norme{x^*+ty^*}-1)$$ and let
$$\modamucw{X}{t,x^*}=\sup_{Y\in\wcof{X^*}}\inf_{y^*\in S_Y}(\max\{\norme{x^*+ty^*},\norme{x^*-ty^*}\}-1).$$  Then let $\modconvunifw{X}{t}=\inf_{x^*\in S_{X^*}}\modconvunifw{X}{t,x^*}$ and let $\modamuc{X}{t}=\inf_{x^*\in S_{X^*}}\modamucw{X}{t,x^*}$. We say that $X^*$ is AUC$^*$ (weak$^*$ asymptoticaly uniformly convex) if $\modconvunifw{X}{t}>0$ for all $t>0$ and we say that $X^*$ is AMUC$^*$ (weak$^*$ asymptotically midpoint uniformly convex) if $\modamucw{X}{t}>0$ for all $t>0$. Note that we clearly have $\widehat{\delta}^*\geq \bar{\delta}^*$ so that any AUC$^*$ norm is AMUC$^*$. The notation comes from  \cite{JLPS} for the first modulus and the second is a weak$^*$ version of the one introduced in \cite{lenses}. Again we do not give the weak version of those moduli and the corresponding results to avoid repetitions.

 \begin{rmk}
 
 By a straightforward convexity argument, one can show that $$\modconvunifw{X}{t,x^*}=\sup_{Y\in\wcof{X^*}}\inf_{y^*\in Y,\ \norme{y^*}\geq 1}(\norme{x^*+ty^*}-1)$$ and  $$\modamucw{X}{t,x^*}=\sup_{Y\in\wcof{X^*}}\inf_{y^*\in Y,\ \norme{y^*}\geq 1}(\max\{\norme{x^*+ty^*},\norme{x^*-ty^*}\}-1).$$
 
 \end{rmk}
 
The following result is well known for AUC$^*$.

\begin{lm}

Let $X^*$ be a dual space with property AUC$^*$. Then for all $t>0$, there exists $\delta>0$ such that for all $x^*\in S_{X^*}$ and for every weak$^*$-null sequence $(x^*_k)_{k\geq 1}$ with $\norme{x^*_k} = t$, we have $\limsup \norme{x^*+x^*_k}> 1+\delta.$

\end{lm}

The proof rests on the well known fact that if Y is a weak$^*$ closed subspace of finite codimension of $X^*$ and if $(x_k^*)_{k\geq 1}$ is a bounded weak$^*$-null sequence in $X^*$ then the distance $\dist(x_k^*,Y)$ goes to 0. In the same way one can prove the following result.

\begin{lm}

Let $X^*$ be a dual space with property AMUC$^*$. Then for all $t>0$, there exists $\delta>0$ such that for all $x^*\in S_{X^*}$ and for all weak$^*$-null sequence $(x^*_k)_{k\geq 1}$ with $\norme{x^*_k} = t$, we have $\limsup \max\{\norme{x^*+x^*_k},\norme{x^*-x^*_k}\}> 1+\delta.$

\end{lm}

\begin{rmk} As before, one can put $\norme{x^*_k} \geq t$ in the two preceding results and we can simply take $\delta=\modconvunifw{X}{t}$ respectively $\delta=\modamucw{X}{t}$. Moreover the converse of the implication holds in both cases under the assumption that $X$ is separable and the moduli are then the best possible $\delta$'s.

\end{rmk}

It is also well known that property AUC$^*$ is equivalent to the so called UKK$^*$ (weak$^*$ uniform Kadec Klee) property in the form introduced in \cite{UKK}. This is the following result.

\begin{lm}\label{UKK}

A dual space $X^*$ is AUC$^*$ if and only if it satisfies the following property: for all $\varepsilon>0$ there is a $\Delta>0$ such that for all $x^*\in B_{X^*}$, if $\norme{x}>1-\Delta$ then there is a weak$^*$-neighborhood of $x^*$ such that $\diam{(V\cap B_{X^*})}\leq \varepsilon$.

\end{lm}

A direct consequence of this result is that if the dual of a Banach space $X$ is AUC$^*$ then the space $X$ has Slzenk index $S_Z(X)\leq \omega$. In fact there is a huge renorming theory around property AUC$^*$ and this is resumed in the following theorem.

\begin{thm}\label{renorming}

Let $X$ be a Banach space. The following assertions are equivalent.

\begin{enumerate}

\item{ The space $X$ satisfies $S_Z(X)\leq \omega$.}

\item{ The space $X$ admits an equivalent norm whose dual norm is AUC$^*$. }

\item{  The space $X$ admits an equivalent norm whose dual norm is $p$-AUC$^*$ for some $1\leq p<\infty$ that is to say there is a constant $c>0$ such that for all $t>0$ we have $\modconvunifw{X}{t}\geq ct^p$. }

\end{enumerate}

\end{thm}

This result has a long history and we refer to \cite{surveySzlenk} or \cite{JLPS} for more informations and references. It was conjectured by R. Huff in \cite{Huff} and then first proved in \cite{KOS} by H. Knaust, E. Odell and Th. Schlumprecht in the separable case  with the terminology UKK$^*$.

We have a corresponding result to the Lemma \ref{UKK} above for property  AMUC$^*$ with a sort of symmetric UKK$^*$ property.

\begin{lm}

A dual space $X^*$ is AMUC$^*$ if and only if it satisfies the following property: for all $\varepsilon>0$ there is a $\Delta>0$ such that for all $x^*\in B_{X^*}$, if $\norme{x^*}>1-\Delta$ then there is a weak$^*$-neighbourhood $V$ of $x^*$ such that for all $y^*\in X^*$ of norm $\norme{y^*}\geq \varepsilon$ we have $x^*+y^*\notin V\cap B_{X^*}*$ or $x^*-y^*\notin V\cap B_{X^*}*$.

\end{lm}

\begin{proof}

Before starting the proof let us mention the following facts which are key to pass from weak$^*$ closed subspaces of finite codimension to weak$^*$ neighbourhoods of $0$ in a dual space and vice versa. 

\begin{enumerate}

\item{ Every weak$^*$ neighbourhood of $0$ in $X^*$ contains a weak$^*$ closed subspace of finite codimension of $X^*$.}

\item{ Every weak$^*$ closed subspace of finite codimension of $X^*$ can be seen as a sort of limit of weak$^*$ neighbourhoods of $0$. More precisely we can find weak$^*$ neighbourhoods of $0$ whose elements are all at the same arbitrarily small distance from the subspace. }

\end{enumerate}

This essentially comes from the fact that every weak$^*$ closed subspacs of finite codimension of a dual can be written as the finite intersection of kernels of elements of the predual.

Now let us prove the proposition. First let us assume that $X^*$ is not AMUC$^*$. Using the first fact above we immediately get that there is a $t>0$ such that for every $\delta>0$ we can find some $x^*=x^*_\delta\in S_{X^*}$ satisfying: $$ \forall W\in\mathcal{V}_{w^*}(0),\ \exists w^*\in W, \ \norme{w^*}\geq t:\ \norme{x^*\pm w^*}<1+\delta.$$ 
Let $\varepsilon=t$, fix $\Delta>0$, pick $\delta>0$ small enough for our later purpose and take $x^*=x^*_\delta\in S_{X^*}$ as above. Then let $u^*=\frac{1}{1+\delta}x^*$ and fix some weak$^*$ neighbourhood $V$ of $u^*$. Take $W\in\mathcal{V}_{w^*}(0)$ such that $V=u^*+W$ and apply the preceding property to find some $w^*$ in $W$ of norm $\norme{w^*}\geq t$ such that $\norme{x^*\pm w^*}<1+\delta$. Finally let $v^*=\frac{1}{1+\delta}w^*$. Then $u^*+v^*\in V\cap B_{X^*}$ and $u^*-v^*\in V\cap B_{X^*}$. Moreover if $\delta$ is small enough then  we have $\norme{u^*} = \frac{1}{1+\delta} > 1-\Delta$ and $\norme{v^*}\geq \frac{\varepsilon}{1+\delta} \geq \frac{\varepsilon}{2}$. Thus $X^*$ does not satisfy the symmetric UKK$^*$ property.

Second let us assume that there is an $\varepsilon>0$ such that for every $\Delta>0$ there is some $x^*=x^*_\Delta$ of norm $\norme{x^*}> 1-\Delta$ satisfying: $$\forall V\in \mathcal{V}_{w^*}(x^*):\ \exists y^*\in X^*,\ \norme{y^*}\geq \varepsilon,:\  x^*+y^*\in V\cap B_{X^*} \text{ and }x^*-y^*\in V\cap B_{X^*}.$$
Let $t=\varepsilon$, fix $\delta>0$, pick $\Delta>0$ small enough for our later purpose and take $x^*=x^*_\Delta$ as above. Then take $Z\in\wcof{X^*}$ and pick $\mu>0$ small enough for our later purpose. Using the second fact above we can choose a weak$^*$ neighbourhood of $0$ whose elements are all  at distance less than $\mu$ from $Z$. Translating, applying the preceding property, and finally projecting onto $Z$ yields a $z^*\in Z$ of norm $\norme{z^*}\geq t-\mu$ such that $\norme{x^*\pm z^*}\leq 1+\mu$.  So if $\Delta$ and $\mu$ are small enough we have $\norme{\frac{x^*}{\norme{x^*}}\pm\frac{z^*}{\norme{x^*}}}\leq \frac{1+\mu}{1-\Delta} < 1+\delta$ and $\norme{\frac{z^*}{\norme{x^*}}}\geq \frac{t-\mu}{1-\Delta}\geq\frac{t}{2}$. Thus $X^*$ is not AMUC$^*$.

\end{proof}

Again there is a sequential version of this result in the separable case.

\begin{lm}

Let $X$ be a separable Banach space. Then $X^*$ is AMUC$^*$ if and only it satisfies the following property: for all $\varepsilon>0$, there exists $\delta\in(0,1)$ such that for all $x^*\in X^*$, if $\norme{x^*}> 1-\delta$ then for all weak$^*$-null sequence $(x_k^*)_{k\geq 1}$ in $X^*$ such that $\norme{x_k^*}\geq \varepsilon$ for all $k\geq 1$ there is a $k_0\geq 1$ such that $\norme{x^*+x_{k_0}^*}> 1$ or $ \norme{x^*-x_{k_0}^*}> 1$.

\end{lm}

As we already mentioned in the preceding section the Szlenk index  has a tight relation with some tree behaviour in the dual of the considered space. In particular, in a dual space $X^*$ with property AUC$^*$, all the weak$^*$-null separated trees of a large enough height must get out of the unit ball.   We can show the same kind of results for AMUC$^*$.

\begin{prp}\label{star}

Let $X^*$ be AMUC$^*$. Then $X^*$ has property $\neg (*)$ that is to say: for all $\varepsilon>0$ there is an integer $N\geq 1$ such that: for all $n\geq N$ and for all weak$^*$-null tree $(x_s^*)_{s\in T_n}$ in $X^*$ such that $x_\emptyset^*\in B_X^*$ and $\norme{x_s^*}\geq \varepsilon$ for all $s\in T_n$ non-empty, there is a sequence $s\in T_n$ of maximal length and a sequence of signs $(\varepsilon_i)_{i=1}^n$ in $\{-1,1\}$ such that $\norme{ x_\emptyset^*+ \sum_{i=1}^n\varepsilon_i x_{s_{\lvert i}}^*} > 1$.

\end{prp}

\begin{proof}

Let us assume that $X^*$ fails this property. Then there exists $\varepsilon>0$ and infinitely many $n$'s for which there exists a weak$^*$-null tree $(x^*_s)_{s\in T_n}$ in $X^*$ such that $\norme{x^*_s}\geq \varepsilon$ for all $s\in T_n$ non-empty and $\norme{ x^*_\emptyset+ \sum_{i=1}^n\varepsilon_i x^*_{s_{\lvert i}}} \leq 1$ for all $s$ in $T_n$ of maximal length and for all choices of signs. Fix such an $n$ big enough for our later purpose and fix such a weak$^*$ null tree. Note that by weak$^*$ lower continuity of the norm, one has  in fact $\norme{ x^*_\emptyset+ \sum_{i=1}^k\varepsilon_i x^*_{s_{\lvert i}}} \leq 1$ for all $1\leq k\leq n$, $s$ of length $k$ and choices of signs. Now let us assume that $X^*$ is AMUC$^*$. Then there is a $\delta=\delta(\varepsilon)>0$ such that for all $x\in S_{X^*}$ and for all $(x^*_k)_{k\geq 1}$ weak$^*$ null with $\norme{x^*_k}\geq \varepsilon$ one has $\limsup \max\{ \norme{x^*+x^*_k},\norme{x^*-x^*_k}\} > 1+\delta.$ In particular, for all $x^*\neq 0\in B_X$ and for all $(x^*_k)_{k\geq 1}$ weak$^*$-null with $\norme{x^*_k}\geq \varepsilon$ one has $\limsup \max\{ \norme{x^*+x^*_k},\norme{x^*-x^*_k}\} > \norme{x^*}(1+\delta).$ First assume that $\norme{x^*_\emptyset}\geq \frac{1}{2}$. Applying successively this property in our weak$^*$-null tree we can get a sequence $s\in T_n$ of maximal length and a sequence of signs $(\varepsilon_i)_{i=1}^n$ in $\{-1,1\}$ such that $\norme{ x^*_\emptyset+ \sum_{i=1}^n\varepsilon_i x^*_{s_{\lvert i}}} > \norme{x^*_\emptyset}(1+\delta)^n \geq\frac{1}{2}(1+\delta)^n $. Second if $\norme{x^*_\emptyset}< \frac{1}{2}$ then one can apply the preceding result with $\tilde{x^*_\emptyset}=-(1-\norme{x^*_\emptyset})\frac{x^*_\emptyset}{\norme{x^*_\emptyset}}$ and get a sequence $s\in T_N$ and a sequence of signs such that $$\norme{ x^*_\emptyset+ \sum_{i=1}^n\varepsilon_i x^*_{s_{\lvert i}}} \geq \norme{ \tilde{x^*_\emptyset}+ \sum_{i=1}^n\varepsilon_i x^*_{s_{\lvert i}}}-\frac{1}{\norme{x^*_\emptyset}}\geq \frac{1}{2}\left((1+\delta)^n-1\right).$$ Clearly this is bigger than $1$ if $n$ was chosen big enough. A contradiction.

\end{proof}

In view of this result it is quite clear that if one restricts to the case of duals of separable spaces with a weak$^*$ unconditional asymptotic structure one gets the following renorming theorem.

\begin{thm}

Let $X$ be a separable space and let us assume that $X^*$ has a weak$^*$ unconditional asymptotic structure. Then the following properties are equivalent. 

\begin{enumerate}

\item{ $S_Z(X)\leq \omega$ }

\item{  The space $X$ admits an equivalent norm whose dual norm is AUC$^*$ }

\item{  The space $X$ admits an equivalent norm whose dual norm is AMUC$^*$. }

\end{enumerate}

\end{thm}

\begin{proof}

$(1)\implies (2)$ comes from the renorming theorem for AUC$^*$ and $(2)\implies (3)$ is straightforward in view of the AUC$^*$ and AMUC$^*$ moduli.

For $(3)\implies (1)$ recall that the dual $X^*$ of a separable space with a weak$^*$ unconditional asymptotic structure has the weak$^*$ unconditional finite tree property: there is a constant $C\geq 1$ such that one can extract a  full subtree whose branches are all $C$-unconditional of any weak$^*$-null normalized tree in $X^*$. So if $X^*$ is AMUC$^*$, this combined with the preceding proposition \ref{star} yields the following property:  for all $\varepsilon>0$ there is an integer $N\geq 1$ such that: for all $n\geq N$ and for all weak$^*$-null tree $(x_s^*)_{s\in T_n}$ in $X^*$ such that $x_\emptyset^*\in B_X^*$ and $\norme{x_s^*}\geq \varepsilon$ for all $s\in T_n$ non-empty, there is a sequence $s\in T_n$ of maximal length and a choice of signs $\varepsilon_i$  such that $\norme{ x_\emptyset^*+ \sum_{i=1}^n\varepsilon_i x_{s_{\lvert i}}^*} > 1$ and consequently $$\norme{ x_\emptyset^*+ \sum_{i=1}^n x_{s_{\lvert i}}^*} \geq C^{-1}\norme{ x_\emptyset^*+ \sum_{i=1}^n\varepsilon_i x_{s_{\lvert i}}^*} > C^{-1}.$$ So by normalizing and by using the characterization \ref{Szlenktrees} mentioned in the preceding section we have $S_Z(X)\leq \omega$.

\end{proof}

\subsection{Size of approximate midpoints in dual spaces} Let $M$ be a metric space. For all $x,y\in M$ and for all $\delta>0$ let $$\midp{x}{y}{\delta}=\{z\in M:\ \max\{d(x,z), d(y,z)\} \leq \frac{1+\delta}{2}d(x,y)\}$$ be the set of $\delta$-approximate midpoints between $x$ and $y$. It appears that the study of the size of those sets yields simple arguments for preventing coarse-Lipschitz embeddability of certain Banach spaces into others. The first result of this type is due to Enflo (in an unpublished paper) and can be stated as follows: $L_1$ does not coarse-Lipschitz embed into $\ell_1$. Let us mention that this can also be used to distinguish between $\ell_p$ spaces under coarse-Lipschitz embeddings and that this argument was pushed even further for example in \cite{KR} or  \cite{Kalton}.

In \cite{lenses}, property AMUC was related to the Kuratowski measure of non compactness of approximate midpoint sets in the following way.

\begin{lm}[\cite{lenses} theorem 2.1]

A Banach space $X$ is AMUC if and only if $\alpha\left( \midp{x}{-x}{\delta}\right)$ tends to $0$ as $\delta$ tends to $0$ uniformly on $S_X$.

\end{lm}

Let us recall that the Kuratowski measure of non-compactness $\alpha$ of a set $A$ is the infimum of all $\varepsilon>0$ such that $A$ can be covered by a finite number of sets of diameter less than $\varepsilon$. Let us also point out that this criteria is sort of an asymptotic version of a characterisation of uniform convexity in terms of diameter of approximate midpoint sets (more about that is described in the same paper).

One can show that the same characterization holds for property AMUC$^*$ on a dual space and this yields the following result.

\begin{prp}

Let $X^*$ be a dual space. Then $\norme{.}_{X^*}$ is AMUC if and only if it is  AMUC$^*$.

\end{prp}

\begin{proof}

It is clear from the definition of the moduli that if $\norme{.}_{X^*}$ is AMUC$^*$ then it is also AMUC.

Now let us assume that $\norme{.}_{X^*}$ is not AMUC$^*$. Following [DKR\&co theorem 2.1] we will prove that $\alpha\left( \midp{x^*}{-x^*}{\delta}\right)$ does not tend uniformly to $0$ with $\delta$ on $S_{X^*}$. Since $\norme{.}_{X^*}$ is not AMUC$^*$ we can find some $t\in(0,1)$ such that: $$\forall \delta>0,\ \exists x_\delta^*\in S_Y:\  \forall Y\in\wcof{X^*},\ \exists y^*\in S_Y:\ \max\{ \norme{x_\delta^*+ty^*}, \norme{x_\delta^*-ty^*} \}\leq 1+\delta,$$ that is to say $ty^*\in \midp{x_\delta^*}{-x_\delta^*}{\delta}$.

Fix $\delta>0$ and pick $x^*=x^*_\delta$ satisfying this property. Then one can choose inductively a sequence $(y_k^*)_{k\geq 1}$ in $S_{X^*}$, and a sequence $(y_k)_{k\geq 1}$ in $S_X$ such that \begin{enumerate}

\item{ $\scal{y_k}{y_k^*}\geq  \frac{1}{2}$ and $y_k^*\in \bigcap_{i=1}^{k-1} \ker y_i$ }

\item{ $ty_k^*\in \midp{x^*}{-x^*}{\delta}$. }

\end{enumerate}

Then we have $\norme{ty_k^*-ty_l^*}\geq \scal{y_k}{ty_k^*-ty_l^*}\geq \frac{t}{2}$ for all $k> l$ and thus $\alpha\left( \midp{x}{-x}{\delta}\right)\geq \alpha \left( \{ty_k^*\}_{k\geq 1} \right) \geq \frac{t}{2}$. By the preceding lemma, $\norme{.}_{X^*}$ is not AMUC.

\end{proof}

As a consequence of this simple result we can show that properties AMUC$^*$ and AUC$^*$ are not equivalent up to renorming.

\begin{crl}

There exists a separable Banach space whose dual norm is AMUC$^*$ but which admits no equivalent norm whose dual norm is AUC$^*$.

\end{crl}

\begin{proof}

Let $JT$ be the James tree-space. Maria Girardi proved in \cite{Girardi} that the dual $JT^*$ of $JT$ is AUC. Thus it is AMUC and by the preceding result it is also AMUC$^*$. But since $JT^*$ is not separable while $JT$ is separable, $JT$ does not have any equivalent norm whose dual norm is AUC$^*$.

\end{proof}

\begin{rmk}

Also note that $JT^*$ being AMUC$^*$ it has property $\neg(*)$ that is to say every sufficiently high weak$^*$-convergent tree in the unit ball of $JT^*$ must get out of it if we choose the right direction. This really enlightens the importance of this choice of direction because the same unit ball contains arbitrarily high separated weak$^*$-convergent trees since $JT$ is separable and has a non-separable dual (and so has infinite Szlenk index).

\end{rmk}

\subsection{A few remarks on the equivalence AUC/AMUC up to renorming} Let us just point out some results about property AMUC that can be obtained in the same way as there weak$^*$ analogues.

\begin{prp}

Let $X$ be AMUC. Then for all $\varepsilon>0$ there is an integer $N\geq 1$ such that: for all $n\geq N$ and for every weakly-null tree $(x_s)_{s\in T_n}$ in $X$ such that $x_\emptyset\in B_X$ and $\norme{x_s}\geq \varepsilon$ for all $s\in T_n$ non-empty, there is a sequence $s\in T_n$ of maximal length and a sequence of signs $(\varepsilon_i)_{i=1}^n$ in $\{-1,1\}$ such that $\norme{ x_\emptyset+ \sum_{i=1}^n\varepsilon_i x_{s_{\lvert i}}} > 1$.

\end{prp}
 
This immediately yields a result concerning the weak-Szlenk index of Banach spaces with property AMUC and an unconditional (weak) asymptotic structure.

\begin{crl}

Let $X$ be a Banach space with separable dual and with an unconditional asymptotic structure. If $X$ is AMUC, then $w-S_Z(X)\leq \omega$. In particular, X has PCP.

\end{crl}

Let us recall that the weak Szlenk index $w-S_Z(X)$ of a Banach space $X$ is obtained by applying the same procedure as for the Szlenk index in the space itself that is by taking away all the small weakly open subsets of the unit ball of $X$ and then iterating.  Unfortunately it is not possible to use this result to get a renorming result for AMUC in this context since there is no renorming theory for AUC in full generality. Still let us just mention that in view of the results from the preceding paragraph one has the following.

\begin{prp}

Properties AMUC and AUC are equivalent up to renorming in the class of duals of separable spaces with a weak$^*$ unconditional asymptotic structure.

\end{prp}

\section{Embeddability of countably branching bundle graphs into dual spaces}
 
In this section we introduce the so called bundle graphs and we prove the main result of the paper by using the notion of good $\ell_p$-trees.

\subsection{Good $\ell_p$-trees and  weak$^*$ asymptotic structure} Since Mazur's work it is a well known fact that it is possible to extract a basic subsequence from every weak or weak$^*$-null normalized sequence. A simple extension of Mazur's proof using the concept of linear compatible ordering allows to prove the tree version of this result whose weak analogue is used and proved for example in \cite{Baudierandco}.

\begin{lm} Let $X$ be a Banach space and let $(x^*_s)_{s\in T_n}$ be a weak$^*$-null normalized tree in $X^*$. For all $\delta\in(0,1)$ there is a full subtree $T$ of $T_n$ and a compatible linear ordering $(\sigma_i)_{i\geq 1}$ of $T$ such that the sequence $(x^*_{\sigma_i})_{i\geq 1}$ is $(1+\delta)$-basic.

\end{lm}

\begin{rmk}

Although not required in this paper let us mention that it is in fact possible to do a bit better in the weak$^*$ case: we can extract a weak$^*$ basic full subtree from every weak$^*$ null normalized tree in the sense of \cite{weakstarbasic} .

\end{rmk}

Now let $X$ be a Banach space, let $(\sigma_i)_{i\geq 1}$ be a linear compatible ordering of $T_n$, let $p\in[1,\ \infty]$ and let $C,D\geq 1$. A normalized tree $(x_s)_{s\in T_n}$ in $X$ is called a $(C,D)$-good $\ell_p$-tree of height $n$ if if satisfies the two following properties:
\begin{enumerate}

\item{ All branches are $C^2$-equivalent to the unit vector basis of $\ell_p^{n+1}$. }

\item{ The sequence $(x_{\sigma_i})_{i\geq 1}$ is $D$-basic. }

\end{enumerate}

We say that $X$ contains good $\ell_p$-trees of arbitrary height almost isometrically if for every $n\geq 1$ and for every $\varepsilon>0$, $X$ contains a $(1+\varepsilon,1+\varepsilon)$-good $\ell_p$-tree of height $n$.

Following \cite{Baudierandco} (in the weak case) we can use the tree version of Mazur's lemma and the tree reformulations from the first section to get the following result.

\begin{lm}

Let $X^*$ be the dual of a separable Banach space $X$ such that $\ell_p^n\in \{X^*\}_n^{w^*}$  for every $n\geq 1$. Then $X^*$ contains good $\ell_p$-trees of arbitrary height almost isometrically.

\end{lm}

Thus, as a consequence of the results of the third section and using the weak analogue from \cite{Baudierandco}, we get the following results.

\begin{crl}

Let $X^*$ be the dual of a separable Banach space $X$ with $S_Z(X)>\omega$, and assume that $X^*$ has a weak$^*$ unconditional asymptotic structure. Then $X^*$ contains good $\ell_\infty$-trees of arbitrary height almost isometrically.

\end{crl}

\begin{crl}

Let $X$ be a Banach space with separable dual, with an unconditional asymptotic structure and with $S_Z(X)>\omega$. Then $X$ contains good $\ell_1$-trees of arbitrary height almost isometrically.

\end{crl}

\subsection{Bundle graphs and embeddability results}

A top-bottom graph is a graph with two distinguished vertices, one designated as the top and the other as the bottom. A countably branching bundle graph is a top-bottom graph which can be formed, starting by a path of length $1$, by a finite sequence of the following operations: 

\begin{enumerate}

\item{ Parallel composition: given two countably branching bundle graphs , identify the top of one with the bottom of the other and let the bottom of the first (respectively top of the second) be the bottom (respectively the top) of the new graph.}

\item{ Series composition: take countably many copies of a countably branching bundle graph and identify all the bottoms (respectively all the tops) with each other. } 

\end{enumerate}

A non-trivial bundle countably branching bundle graph is a bundle graph obtained by such a sequence with at least one series composition. We endow every bundle graph with its graph distance.

One can have a look at \cite{Swift} for a very detailed survey of properties of bundle graphs (in particular an explicit formula for the distance in such a graph is given and some results of embeddability into Banach spaces are proved). In particular, it was proved in this paper that if one consider the operation $\oslash$ which consist in replacing every edge of some countably branching bundle graph by another countably branching bundle graph, one gets a new countably branching bundle graph. Thus, for every countably branching bundle graph $G$, one can consider the family $(G^{\oslash n})_{n\geq 1}$ of countably branching bundle graphs (with growing height) generated by $G$ defined recursively by $G^{\oslash 1}=G$ and $G^{\oslash (n+1)}=G\oslash G^{\oslash n}$.

In \cite{Baudierandco}, the authors used a self improvement argument and an approximate midpoint argument coming from the study of AMUC property to show the following result (with an estimate of the distortion).

\begin{thm}[theorem 4.1 \cite{Baudierandco}]

Let $G$ be a non-trivial countably branching bundle graph and let $X$ be an AMUC Banach space. Then the family $(G^{\oslash n})_{n\geq 1}$ does not equi-Lipschitz embed into $X$.

\end{thm}

Generalizing a result from \cite{Baudierandco} concerning the countably branching diamond graphs, Swift proved the following result.

\begin{thm}\label{Swift}

Let $X$ be a Banach space which contains good $\ell_\infty$-trees of arbitrary height almost isometrically. Then every non-trivial countably branching bundle graph embeds Lipschitz in $X$ with distortion less than $6+\varepsilon$.

\end{thm}

Thus, combining this result with the containment of good $\ell_\infty$-trees coming from \cite{Baudierandco}, he got (\cite{Swift} corollary 3.4) the embeddability of any countably branching bundle graph into a separable reflexive Banach space with an unconditional asymptotic structure and with $S_Z(X^*)>\omega$. By our result from the last section, we can extend this to dual spaces in the following.

\begin{crl}

Let $X^*$ be the dual of a separable Banach space $X$ with $S_Z(X)>\omega$, and assume that $X^*$ has a weak$^*$ asymptotic unconditional structure. Then every non-trivial countably branching bundle graph embeds Lipschitz in $X^*$ with distortion less than $6+\varepsilon$.

\end{crl}

Then combining this and the non-embeddability result from \cite{Baudierandco} one obtains the following metric characterization.

\begin{thm}

Let $X^*$ be the dual of a separable Banach space $X$ and assume that $X^*$ has a weak$^*$ asymptotic unconditional structure. Also let $G$ be any non-trivial countably branching bundle graph. Then the following properties are equivalent.
\begin{enumerate}

\item{ $S_Z(X)\leq \omega$ }

\item{ $X$ admits an equivalent norm whose dual norm is AUC$^*$ }

\item{ $X$ admits an equivalent norm whose dual norm is AMUC$^*$ }

\item{ The family $(G^{\oslash n})_{n\geq 1}$ does not equi-Lipschitz embed into $X^*$.}

\end{enumerate}

\end{thm}

As a corollary, we get the following stability result.

\begin{crl}

The class of $AUC$ renormable spaces is stable under coarse-Lipschitz embeddings inside the class of duals of separable Banach spaces with a weak$^*$ asymptotic unconditional structure.

\end{crl}

Also note that following the same procedure the remark from section $2$ concerning property $(*)$ yields the following result.

\begin{crl}

Let $X^*$ separable dual with property $(*)$. Then every non-trivial countably branching bundle graph embeds Lipschitz in $X^*$ with distortion less than $6+\varepsilon$.

\end{crl}

\subsection{Bundle graphs and $c_0$-spreading models}  Let us first recall the definition of a spreading model. Let $X$ be a Banach space. By using Ramsey's theorem, one can show that for every bounded sequence $(x_n)_{n\geq 1}\subset X$ there a subsequence $(y_n)_{n\geq 1}$ such that for all $k\geq 1$ and  for all $a_1,\dots, a_k\in \R$ the limit $\lim_{n_1<\dots<n_k} \norme{\sum_{i=1}^ka_iy_{n_i}}$ exists. Let $(e_i)_{i\geq 1}$ be the canonical basis of $c_{00}$. If the sequence $(y_n)_{n\geq 1}$ is not convergent the quantity $$\norme{\sum_{i=1}^ka_ie_i}=\lim_{n_1<\dots<n_k} \norme{\sum_{i=1}^ka_iy_{n_i}}$$ defines a norm on $c_{00}$. The completion of the space $(c_{00},\norme{.})$ is called \emph{spreading model} associated with the fundamental sequence $(e_i)_{i\geq 1}$ and generated by the sequence $(y_n)_{n\geq 1}$. Note that the fundamental sequence is spreading in the sense that for all $k\geq 1$, for all $a_1,\dots,a_k\in \R$ and for all $1\leq n_1<\dots<n_k$ we have the norm equality $\norme{\sum_{i=1}^ka_ie_i}=\norme{\sum_{i=1}^ka_ie_{n_i}}$. If $E$ is another Banach space we shall say that $X$ has an $E$-spreading model if $X$ has a spreading model isomorphic to $E$. We refer to \cite{Beauzamy} for a presentation of the theory of spreading models.

Let us recall the following well known  result concerning $c_0$-spreading models.

\begin{prp}

If $X$ has a $c_0$-spreading model then for every $\varepsilon>0$ there is a normalized sequence $(x_k)_{k\geq 1}$ in  $X$ such that:  \begin{align*}
\forall k\geq 1,\ \forall a_1,\dots,a_k\in\R, \ \forall n_k>...>n_1\geq k, \\
\frac{1}{1+\varepsilon} \max_{1\leq i\leq n} \abs{a_i} \leq \norme{\sum_{i=1}^na_ix_{n_i}}\leq (1+\varepsilon)  \max_{1\leq i\leq n} \abs{a_i} \ \ (\Box)
\end{align*}

\end{prp}

Moreover, let us recall that a $c_0$-spreading model is always generated by a weak-null sequence in $X$. As a consequence it is clear that a Banach space $X$ with property AMUC does not have $c_0$-spreading models. In fact we can even prove the following result.

\begin{thm}

If $X$ has a $c_0$-spreading model, then $\ell_\infty^n\in \{X\}_n$ for all $n\geq 1$. Consequently, every non-trivial countably branching bundle graph embeds Lipschitz in $X$ with distortion less than $6+\varepsilon$.

\end{thm}

\begin{proof}

Fix $n\geq 1$ and $\varepsilon>0$. We can apply the preceding proposition to get a normalized sequence $(x_k)_{k\geq 1}$ satisfying the property $(\Box)$ above. This sequence generates a $c_0$-spreading model and thus it is necessarily weak-null.

For all $s=(s_1,\dots, s_k)\in T_n$ non empty, let $y_s=x_{\phi(s)}$ where $\phi(s)=k+\sum_{i=1}^k s_i$. Then $(x_s)_{s\in T_n\backslash\{\emptyset\}}$ is a weak-null normalized unrooted tree and since $n\leq \phi(s_{\lvert 1})< \dots < \phi(s)$ for every $s\in T_n$ of length $n$, property $(\Box)$ above yields that every branch of this unrooted tree is $(1+\varepsilon)^2$ equivalent to the unit vector basis of $\ell_\infty^n$.

From this we deduce that $\ell_\infty^n\in\{X\}_n$ for every $n\geq 1$ and as before we get that $X$ contains good $\ell_\infty$ trees of arbitrary height almost isometrically. The conclusion follows from \ref{Swift}.

\end{proof}

\section{Questions and comments}

\subsection{Szlenk index of WCG spaces is determined by separable quotients}

It is well known that the Szlenk index is separably determined and this usually allows to get rid of separability assumptions in results concerning it. Unfortunately passing to a separable subspace won't be sufficient to do so in theorem \ref{thmB} and \ref{thmC} since we might lose the asymptotic unconditional property of the dual space in the process. What is required is to pass to a separable quotient while conserving the same information on the Szlenk index and this is obviously way more difficult since it is not even known if a separable quotient exists. Still this can be done if we restrict ourselves to weakly compactly generated (WCG) Banach spaces. Let us recall that a Banach space $X$ is WCG if there is a weakly compact subset $K$ of $X$ such that $\overline{\text{span}}(K)=X$.

\begin{thm}

Let $X$ be a WCG Banach space and let $\alpha$ be a countable ordinal. If $S_Z(X)>\alpha$ then there is a subspace $Y$ of $X$ such that the quotient $Z=X\backslash Y$ is separable and satisfies $S_Z(Z)>\alpha$.

\end{thm}

As we shall see the main reason for working with WCG Banach spaces is that this property passes to quotients and that the density of such a space is equal to the weak$^*$ density of its dual. This can be found in \cite{bookHajek} in the first section of chapter $13$.

We will prove this result by using the $\ell_1^+$ index introduced by D. Alspach, R. Judd and E. Odell in \cite{AJO} which is known to be equal to the Szlenk index for Banach spaces not containing $\ell_1$. Let us recall that a space $X$ has $\ell_1^+$ index $I^+(X)>\alpha$ for some countable ordinal $\alpha$ if and only if there is a normalized tree $(x_s)_{s\in T}$ of order $\alpha$ in $X$ and a constant $K\geq 1$ such that every branch $(x_s)_{s\in\beta}$ of the tree is a $K$-$\ell_1^+$ sequence that is to say it is $K$-basic and satisfies for all sequence $(a_s)_{s\in\beta}$ in $\R_+$ the following inequality $$ \frac{1}{K}\sum_{s\in\beta}a_s\leq \norme{ \sum_{s\in\beta}a_sx_s}.$$

We refer to their paper for precise definitions and for the related result. This tool in hands let us prove the theorem.

\begin{proof}

Let us start by assuming that $X$ does not contain $\ell_1$. Then $I^+(X)=S_Z(X)>\alpha$ and we can find a normalized tree $(x_s)_{s\in T}$ of order $\alpha$ in $X$ witnessing this as recalled above. For every element $s$ in $T$ pick a norming functional $f_s$ of $x_s$ and for every branch $\beta$ of $T$ and every sequence $a=(a_s)_{s\in\beta}$ in $\Q_+$ pick a norming functional $g_{a,\beta}$ of  $\sum_{s\in\beta}a_sx_s$. Then let $F$ be the set of all functionals $f_s$ and $g_{a,\beta}$ and let $Y=F^\top$ be the pre-orthogonal of $F$. It is a well known fact that the dual of the quotient $Z=X\backslash Y$ is isometric to the weak$^*$ closure $\overline{F}^{w^*}$ of $F$ in $X^*$ which is weak$^*$ separable by construction. Now $Z$ is WCG as quotient of a WCG space and thus it is separable since its density is the same as the weak$^*$ density of its dual. Moreover the tree $(\bar{x}_s)_{s\in T}$ will be normalized in $Z$ thanks to the functionals $f_s$ and its branches will all be $K-\ell_1^+$ sequences over $\Q$ thus also over $\R$ thanks to the fonctionals $g_{a,\beta}$. Consequently the $\ell_1^+$ index of $Z$ will also be strictly greater than $\alpha$. Now $Z$ cannot contain $\ell_1$ as a quotient of a space not containing it (it is well known that $\ell_1$ can be lifted from quotients) so $S_Z(Z)=I^+(Z)>\alpha$ and we are done.

Now if $X$ contains $\ell_1$ the same idea allows to build a separable quotient of $X$ which also contains $\ell_1$ and this is enough to conclude.

\end{proof}

This allows to give a non-separable version of theorems \ref{thmB} and \ref{thmC} combined for WCG Banach spaces.

\begin{thm}

Let $X^*$ be the dual of a WCG Banach space $X$ and assume that $X^*$ has a weak$^*$ asymptotic unconditional structure. Also let $G$ be any non-trivial countably branching bundle graph. Then the following properties are equivalent.
\begin{enumerate}

\item{ $S_Z(X)\leq \omega$ }

\item{ $X^*\in \langle AUC^* \rangle$ }

\item{ $X^*\in \langle AMUC^* \rangle$ }

\item{ The family $(G^{\oslash n})_{n\geq 1}$ does not equi-Lipschitz embed into $X^*$.}

\end{enumerate}

\end{thm}

In particular since every reflexive Banach space is WCG the separability assumption from the original theorem from \cite{Swift} can be removed.

\begin{thm}

Let $X$ be a reflexive Banach space with an unconditional asymptotic structure.  Then the following properties are equivalent.
\begin{enumerate}

\item{ $S_Z(X^*)\leq \omega$ }

\item{ $X\in \langle AUC \rangle$ }

\item{ The family $(G^{\oslash n})_{n\geq 1}$ does not equi-Lipschitz embed into $X$.}

\end{enumerate}

\end{thm}

\subsection{About asymptotic midpoint convexity}

\begin{quest}

Is an equivalent AMUC$^*$ norm on a dual space a dual norm?

\end{quest}

It is known to be true for AUC$^*$ norms because such a norm has to be weak$^*$ lower semicontinuous and thus has to be a dual norm (for example see \cite{DGZ} chapter $1$, fact 5.4 ).

\begin{quest}

Is it possible to develop a specific renorming theory for AMUC$^*$?

\end{quest}

As mentioned in the introduction a dual property for AMUC$^*$ or for property $(*)$ is lacking in order to do a renorming à la \cite{KOS} or \cite{GKL}. Moreover it is not clear whether one can find a suitable derivation index to go either with the symmetric UKK$^*$ property we saw in section $3$ or with property $(*)$ or not (or in an other way to find a suitable measure of non-compacteness to go in the direction of \cite{LPR}). It is also not known if it is possible to find a correct equivalent expression of property $(*)$ in order to get a partial result of the type \cite{UKK} based on an Enflo-like renorming to get a symmetric UKK$^*$ écart.

\begin{quest}

If a dual space is AMUC$^*$ then can we find an equivalent $p$-AMUC$^*$ norm for some $1\leq p<\infty$?

\end{quest}

The same question can be asked for property AMUC. Such a result would in particular yield good estimates of the distortion of the families generated by a countably branching bundle graph into an AMUC space as mentioned in \cite{Baudierandco}

\begin{quest}

If a space is AMUC does it have weak-Slzenk index $w-S_Z(X)<\infty$? Does it have weak-Szlenk index $\omega$?

\end{quest}

Again we saw that this is true when the space has a separable dual and an unconditional asymptotic structure. The first question can be reformulated in the following way: if a space is AMUC does it have PCP? In relation with the comments of the following paragraph one can also ask whether AMUC implies weaker versions of PCP like convex PCP or not. A negative answer to any of these question would provide a counter example to the question of the equivalence between AMUC and AUC up to renorming since since an AUC space has finite weak-Szlenk index. Let us point out that \cite{lenses} gave partial positive answers to this question in the case of Banach spaces with an unconditional basis and that they settled the question of isometric equivalence by constructing an equivalent norm on $\ell_2$ which is AMUC but not AUC. Let us also point out that there is no renorming theory around  weak-Szlenk index yet (the obstruction is always a lack of weak compactness).

\subsection{About James tree-spaces}

\begin{quest}

Is it possible to improve the exponent for AUC given in \cite{Girardi} (3-AUC) for the predual of the James tree-space $JT$ on a dyadic tree? What exponent can we get for its dual? In view of the definition of this James tree-space one expects to get an exponent $2$.

\end{quest}

\begin{quest}

Is the predual of the James-tree space $JT_\infty$ on a countably branching tree AMUC? Is its dual AMUC$^*$? This would provide a counter-example to several questions of the preceding paragraph since the predual of $JT_\infty$ does not have PCP. Let us point out that this space still has convex PCP (see  \cite{GMS}).

\end{quest}

\bibliographystyle{plain}
\bibliography{biblio}
\nocite{*}

\end{document}